\documentclass[11pt, openright]{article}
\usepackage{latexsym}
\usepackage[usenames, dvipsnames]{color}          %--- Pacchetto per il testo colorato
\usepackage{amsfonts}                %--- Pacchetto per i font speciali (R, Q, Z...)
\usepackage{amsmath}                 %--- Pacchetto aggiuntivo per font matematici
\usepackage[T1]{fontenc}
\usepackage{amsthm}         %--- Pacchetto per l'uso di theoremstyle
\usepackage{amssymb}        %--- Pacchetto per l'uso dei simboli AMS
\usepackage{graphicx}
\usepackage{float}
\usepackage{subfigure}

\usepackage{varioref}

\usepackage[mathcal]{euscript} % --- Pacchetto per nuovi simboli \mathcal
\usepackage{mathrsfs}
\usepackage{newlfont}          % --- Pacchetto per utilizzare diversi font (fourier, ecc.)

\theoremstyle{plain}
\newtheorem{Th}{Theorem}
\newtheorem{Prop}[Th]{Proposition}
\newtheorem{Cor}[Th]{Corollary}
\newtheorem{Lm}[Th]{Lemma}
\theoremstyle{definition}
\newtheorem{Def}[Th]{Definition}

\theoremstyle{remark}
\newtheorem{Oss}[Th]{Remark}
\newcommand{\Lieg}{\mathfrak{g}}
\newcommand{\ad}{\mathrm{ad \ }}

\newcommand{\adj}[1]{\mathrm{ad}\,#1}

\newcommand{\Lie}{\mathcal{L}_{\mathbb{R}}(x,y)}

\newcommand{\tild}[1]{\widetilde{#1}}
\newcommand{\Exp}{\mathrm{Exp}}

\renewcommand{\d}{\mathrm{d}}

\author
{
  Stefano Biagi \\
  Alma Mater Studiorum - Università di Bologna
}
\title
{
  On M.\,Mérigot's theorem on the convergence domain
  \\
  of the Campbell-Baker-Hausdorff-Dynkin series
}
\begin{document}
\maketitle
\begin{abstract}
 The aim of this manuscript is to present the proof given by Michel Mérigot in 1974 for an enlarged convergence domain of the Campbell-Baker-Hausdorff-Dynkin series
 in the Lie algebra of a Banach-Lie group. This proof is based on a theorem, of independent interest, 
 on the lifetime of the solution of a Cauchy problem. We furnish all the details for this ODE result in Appendix A.
\end{abstract}

\section{Introduction}
 The aim of this manuscript is to present the proof given by Michel Mérigot in 1974 for an enlarged convergence domain of the Campbell-Baker-Hausdorff-Dynkin series (we follow the naming convention in the recent monograph \cite{BonfiglioliFulci}, and we abbreviate it as CBHD series) in the Lie algebra of a Banach-Lie group. This elegant proof is based on a theorem, of independent interest, on the lifetime of the solution of a Cauchy problem.
 For a more recent result on the convergence domain of the CBHD series, see Blanes, Casas \cite{BlanesCasas}.
 See also \cite{BiagiBonfiglioli} for the case of infinite-dimensional Banach-Lie algebras not coming from Lie groups.

  The first section of this manuscript is totally devoted to describe the argument by Mérigot, as appeared in his unpublished manuscript \cite{Merigot}. Later on, we insert two 
  Appendices, \emph{not} appearing in the original work of Mérigot, where we provide more details.\medskip

 The material presented here will be part of the PhD Thesis ``Subelliptic Operators on Lie Groups'' by Stefano Biagi 
 (in preparation, 2013; Department of Mathematics, Alma Mater Studiorum Università di Bologna: 
 Bologna, Italy. Advisor: Prof.\,Andrea Bonfiglioli).\medskip

 \textbf{Acknowledgements.}
 The availability of the manuscript \cite{Merigot} is due to
 Professor Jean Michel, who shared the original copy with the Advisor of my PhD Thesis.
 I join my Advisor in thanking Professor Michel for making available to us the manuscript.

\section{The convergence domain of the Campbell-Hausdorff series.} \label{Sec::1}
\begin{quote}
[The contents of Section \ref{Sec::1} contain a description of the arguments in \cite{Merigot}, as faithful as possible if compared to the original manuscript. We restrict to present the argument concerning the \emph{`Campbell-Hausdorff series'} (as Mérigot calls it), highlighting in due course the points where we perform any omission. Furthermore, we allow ourselves to insert some clarifications on definitions and notations (in blue footnotes).]
\end{quote}
 \bigskip

Let $\mathbb{G}$ be a Banach-Lie group and let $\Lieg := \mathrm{Lie}(\mathbb{G})$ be the Lie-algebra of $\mathbb{G}$. From the fact that $\mathbb{G}$ is a Banach-Lie group, it follows that $\Lieg$ is a Banach space, and the norm on this space is compatible with the Lie algebra structure:
\begin{equation}\label{submult}
\|[a,b]\| \leq \|a\|\|b\|, \qquad \forall \ a,b \in \Lieg.
\end{equation}
 For $a,b$ in a suitable neighborhood of $0$ (the neutral element of $\Lieg)$ it is defined the so-called \emph{Hausdorff function}
\begin{equation}\label{Hfunct}
  \phi(a,b)=a+b+\frac{1}{2}\,[a,b]+
  \frac{1}{12}([a,[a,b]]+[b,[b,a]])-\frac{1}{24}\,[a,[b,[a,b]]]+\cdots,
\end{equation}
 and the (partial) map $y \mapsto \phi(a,b)$ satisfies the following Cauchy problem\footnote{\color{blue} In the sense of Fréchet differentials in the Banach space $\Lieg$.}
\begin{equation} \label{S1IntCauchypb}
\begin{cases}
 \displaystyle \d_b\phi(a,b) = \frac{\adj{\phi(a,b)}}{1 - \exp(-\adj{\phi(a,b)})}
 \circ\frac{1 - \exp(-\adj{b})}{\adj{b}} \\
\phi(a,0) = a.
\end{cases}
\end{equation}
%%%%%%%%%%%%%%%%%%%%%%%%%%%%%%%%%%%%%%%%%%
\subsection{Properties of the function $\phi$.}
First of all, we recall that the function $\phi$ satisfies the following identity
$$\Exp(\phi(a,b)) = \Exp(a)\,\Exp(b),$$
where $\Exp:\Lieg \to \mathbb{G}$ is the Exponential Map.
Let now\footnote{
\color{blue} Here and in the sequel, we adopt the usual convention in defining functions
of continuous endomorphisms of the vector space $\Lieg$. For example,
$$\frac{\exp(\adj{z})-1}{\adj{z}} := \sum_{n = 0}^{\infty}\frac{(\adj{z})^{n}}{(n+1)!},\quad z\in \Lieg.$$
}
$$\mathcal{A} = \left\{z \in \Lieg: \frac{1 - \exp(\adj{z})}{\adj{z}} \ \textrm{is invertible}\right\}.$$
 We define $\Delta \subseteq \Lieg\times\Lieg$ to be the set of the couples $(a,b)$ such that $\phi(a,b)$ is well-posed as the maximal solution of \eqref{S1IntCauchypb}. The set $\Delta$ is an open subset of $\Lieg\times\Lieg$, and
$$\phi: \Delta \longrightarrow \mathcal{A}.$$
 For example, if $a,b \in \Lieg$ are such that $[a,b] = 0$ and $a + tb \in \mathcal{A}$ for all $0 \leq t \leq 1$, then $(a,b) \in \Delta$ and $\phi(a,b) = a + b$.
\begin{Prop}\label{prop.bern}
The set $\mathcal{A}$ contains the open ball ${B}$ in $\Lieg$ centered at $0$ and radius $2\pi$.
\end{Prop}
\begin{proof}
The function
$$f(w) = \frac{w}{1 - e^w}, \quad w \in \mathbb{C}$$
is holomorphic on $D = \{w\in \mathbb{C} : |w| < 2\pi\}$. Its Maclaurin series is therefore convergent on $D$, uniformly on every compact subset $K \subseteq D$, and\footnote{
\color{blue} We remark that $\alpha_n=-B_n/n!$, where the $B_n$ are the usual Bernoulli numbers.
}
$$f(w) = \sum_{n = 0}^{\infty}\alpha_n\,w^n, \quad w \in D.$$
For $z$ in a suitable neighborhood of $0$ in $\Lieg$, let us now consider the function $\tild{f}(z)$ defined by the power series
$$\widetilde{f}(z):=\sum_{n = 0}^{\infty}\alpha_n(\adj{z})^n.$$
A sufficient condition for this power series to be
convergent is that\footnote{\textcolor[rgb]{0.00,0.00,1.00}{The original manuscript uses the notation $\|\cdot\|$ for the operator norm of $\adj{z}$; to avoid any confusion with the norm of the Banach space $\Lieg$, we indicate the operator norm by $\big\bracevert\cdot\big\bracevert$, that is
$$\big\bracevert \adj{z}\big\bracevert :=\sup_{h\in \Lieg :\, \|h \|\leq 1} \|(\adj{z})(h)\|. $$}}
$\big\bracevert \adj{z}\big\bracevert < 2\pi$, and thus, in particular, it is convergent if $\|z\| < 2\pi$
(i.e., if $z\in B$), since
(see \eqref{submult})
$$\big\bracevert\adj{z}\big\bracevert \leq \|z\|.$$
By definition, one has
$$\frac{1 - \exp(\adj{z})}{\adj{z}} = -\sum_{n = 0}^{\infty}\frac{(\adj{z})^{n}}{(n+1)!}, \quad z\in \Lieg.$$
From the fact that
$$\frac{1 - e^w}{w}\cdot f(w) = 1,\quad w\in D,$$
we obtain the same kind of identity for $\tild{f}$:
$$\frac{1 - \exp(\adj{z})}{\adj{z}}\circ \tild{f}(z) = \textrm{id}, \quad \forall \ z \in {B},$$
and this proves the proposition.
\end{proof}
\begin{quote}
 \textcolor[rgb]{0.00,0.00,1.00}{[Now we omit, from the dissertation in \cite{Merigot},
a proposition concerning some invertibility properties of
the Exponential Map from $\Lieg$ to $\mathbb{G}$, which is not
used in the sequel.]}
\end{quote}
%%%%%%%%%%%%%%%%%%%%%%%%%%%%%%%%%%%%%%%%%%%%%%%%%%%%%%%%%%%%%%%%%%%%%%%
\subsection{Study of the lifetime of the solution of an ODE.}\label{Sec::2}
First of all, we introduce some notations.
Henceforth, $(E,\|\cdot\|)$ is a fixed real Banach space.
Let $a,b \in \mathbb{R}$ be two fixed positive real numbers, and let $I = [0,a]$ and $J = [0,b]$. We set
$$\tild{\Omega} := I\times J,\quad\text{and}\quad
 \Omega := \Big\{(t,y)\in \mathbb{R}\times E : (|t|,\|y\|) \in \tild{\Omega}\Big\}.$$
\begin{Prop} \label{S1SS1MainProp}
Let us consider the $E$-valued differential equation $y' = f(t,y)$, defined on $\Omega$, where
$f(t,y)$ is a locally Lipschitz-continuous $E$-valued function defined on $\Omega$. Moreover, let us assume that
\begin{equation}\label{ineg}
\|f(t,y)\| \leq g(|t|,\|y\|),\quad \forall\,\,(t,y)\in \Omega,
\end{equation}
 where $g(t,z)$ is a locally Lipschitz-continuous function on $\tild{\Omega}$ which is
 non-decreasing w.r.t.\,the second variable $z\in J$. We use the notation $y = \varphi(t)$, $t \in (\alpha,\beta)$, for the maximal solution of the Cauchy problem (in the Banach space $E$)
\begin{equation}\label{EDOfifi}
 \begin{cases}
 y' = f(t,y) \\
 y(0) = x,
 \end{cases}
\end{equation}
 and the notation
$\tild{y} = \tild{\varphi}(t)$, $t \in (\tild{\alpha},\tild{\beta})$, for
the maximal solution of the real Cauchy problem
$$\begin{cases}
\tild{y}' = g(t,\tild{y}) \\
\tild{y}(0) = \|x\|.
\end{cases}$$
Then one has $$(-\tild{\beta},\tild{\beta}) \subseteq (\alpha,\beta).$$
\end{Prop}
In order to prove this proposition, we remark that is not restrictive to assume that $\|x\| = 0$
(see Section \ref{Sec::Appendix1}). With this assumption, let us first show that
$$\|\varphi(t)\| \leq \tild{\varphi}(t), \quad \text{for every
 $t \in K := (0,\tild{\beta})\cap(0,\beta)$.}$$
 This is a consequence of the following results: Lemmas
  \ref{S1SS1Lm1} and \ref{S1SS1Lm2}.
 In the first preliminary lemma, the hypothesis
 \eqref{ineg} of Proposition \ref{S1SS1MainProp} is strengthened to a strict inequality.
\begin{Lm} \label{S1SS1Lm1}
Let us consider the differential equations
\begin{align}
y' = h(t,y), & \quad \textrm{$h$ an $E$-valued function, defined and continuous on $\Omega$,} \label{S1SS1Eq1Lm1}\\
z' = \tild{h}(t,z), & \quad \textrm{$\tild{h}$ a real-valued function, defined and continuous on $\tild{\Omega}$}, \label{S1SS1Eq2Lm1}
\end{align}
and let us assume that the \emph{strict inequality}
\begin{equation}\label{ipotesistretta}
 \|h(t,y)\| < \tild{h}(|t|,\|y\|)
\end{equation}
is fulfilled, for every $(t,y)\in \Omega$.
We assume that $\widetilde{h}(t,z)$ is non-decreasing w.r.t.\,$z\in J$.
Moreover, let $\psi(t)$ and $\tild{\psi}(t)$ be the maximal solutions of the problems \eqref{S1SS1Eq1Lm1} and \eqref{S1SS1Eq2Lm1} such that
$$\psi(0)=0\in E,\qquad \tild{\psi}(0) = 0\in\mathbb{R}.$$
We denote by $(\alpha,\beta)$ and by $(\widetilde{\alpha},\widetilde{\beta})$
the maximal domains of $\psi$ and $\widetilde{\psi}$, respectively.\medskip

Then, for every $t \in K:=(0,\min\{\beta,\widetilde{\beta}\})$, one has
 $$\|\psi(t)\| < \tild{\psi}(t).$$
\end{Lm}
\begin{proof}
Since, by hypothesis, one has
$$\|h(0,0)\| < \tild{h}(0,0),$$
and the functions $h$ and $\tild{h}$ are both continuous, there exist $\eta, \tau > 0$ such that
$$\|h(t,y)\| < \tild{h}(t,z),
\quad \text{whenever $\|y\|,\,|z| < \eta$  and $|t| < \tau$.}$$
Starting from this inequality, and by applying the well-known Mean Value Theorem to the functions
 $\psi$ and $\tild{\psi}$, one can find\footnote{\color{blue} See Section \ref{Sec::Appendix1}
 for more details.} a real number $\tau_1$ such that
\begin{equation}\label{tau1}
\|\psi(t)\| < \tild{\psi}(t), \quad \forall \ |t| < \tau_1.
\end{equation}
Now, let us assume that $t_0 \in K$ is such that $\|\psi(t)\| < \tild{\psi}(t)$ for all $t < t_0$; then
 we show that this holds true up to $t_0$ itself, i.e.,
$$\|\psi(t_0)\| < \tild{\psi}(t_0).$$
Indeed, since the functions $\psi$ and $\tild{\psi}$ are solutions of the Volterra identities
$$\psi(t) = \int_0^th(u,\psi(u))\,\d u, \quad \tild{\psi}(t) = \int_0^t\tild{h}(u,\tild{\psi}(u))\,\d u,$$
then, for all $0 \leq t < t_0$, one gets
\begin{align}
\|\psi(t)\| & = \left\|\int_0^th(u,\psi(u))\,\d u\right\| \nonumber \\
& < \int_0^t\tild{h}(u,\|\psi(u)\|)\,\d u \label{S1SS1EqProofLm1_1}\\
& \leq  \int_0^t\tild{h}(u, \tild{\psi}(u))\,\d u = \tild{\psi}(t). \label{S1SS1EqProofLm1_2}
\end{align}
Inequality \eqref{S1SS1EqProofLm1_1} follows from \eqref{ipotesistretta}, while
inequality \eqref{S1SS1EqProofLm1_2} holds true since the map $\tild{h}(u,z)$ is increasing w.r.t.\,$z$ (and due to
the assumption on $t_0$). For $t = t_0$, the inequality \eqref{S1SS1EqProofLm1_1} is
still true in the strict sense, and thus
$$\|\psi(t_0)\| < \tild{\psi}(t_0).$$
Since this inequality is fulfilled
in a neighborhood of $0$ (see \eqref{tau1}), the above ``prolongation''
argument demonstrates that the inequality can be actually extended to the whole of $K$, and this proves the lemma.
\end{proof}
\begin{Lm} \label{S1SS1Lm2}
 With the same notation and hypotheses
  of Lemma \ref{S1SS1Lm1}, if $\|h(t,y)\| \leq \tild{h}(|t|,\|y\|)$ then one has
 $$\|\psi(t)\| \leq \tild{\psi}(t)\quad \text{for every $t\in [0,\min\{\beta,\widetilde{\beta}\})$}.$$
\end{Lm}
\begin{proof}
 Given a real parameter $\epsilon>0$, let us consider the differential equations
\begin{align*}
y' & = h(t,y), \\
\tild{y}' & = \tild{h}(t,\tild{y}) + \epsilon\,(1 + \tild{y}) =: \tild{h}_{\epsilon}(t,\tild{y}),
\end{align*}
 and let $\psi$ and $\tild{\psi}_{\epsilon}$ be their maximal solutions such that $\psi(0)=0\in E$
  and  $\tild{\psi}_{\epsilon}(0) = 0\in\mathbb{R}$, with maximal domains
  $(\alpha,\beta)$ and $(\widetilde{\alpha}_\epsilon, \widetilde{\beta}_\epsilon)$. As usual, when
  $\epsilon=0$ we simply write $(\widetilde{\alpha}_0, \widetilde{\beta}_0)=
  (\widetilde{\alpha}, \widetilde{\beta})$.

   For all fixed $t \in K=(0,\min\{\beta,\widetilde{\beta}\})$, there exists $\epsilon_1$ (depending on $t$) such that
$$t \in [0,\tild{\beta}_{\epsilon}), \quad \forall \ \epsilon < \epsilon_1.$$
 (This is a consequence of the continuity of $\widetilde{h}_\epsilon$
  w.r.t.\,$\epsilon$ and standard ODE Theory results.)
 Since $\epsilon>0$, it follows from Lemma \ref{S1SS1Lm1} that $\|\psi(t)\| < \tild{\psi}_{\epsilon}(t)$, and since
 $\lim_{\epsilon \to 0^+}\tild{\psi}_{\epsilon}(t) =\tild{\psi}(t)$,
 one has $\|\psi(t)\| \leq \tild{\psi}(t)$.
\end{proof}
 We are finally in a position to provide the following:

 \begin{proof}[Proof (of Proposition \ref{S1SS1MainProp}.)]
 We inherit all notations in the statement of Proposition \ref{S1SS1MainProp}.
  By contradiction, let us assume that $\beta < \tild{\beta}$
  (thus $\min\{\beta,\widetilde{\beta}\}=\beta$). The function $\varphi$ is defined on $(\alpha,\beta)$
  and the function $\widetilde{\varphi}$ is defined on $(\widetilde{\alpha},\widetilde{\beta})$, and from Lemma \ref{S1SS1Lm2} it follows that
 \begin{equation}\label{diseg.qqqq}
 \|\varphi(t)\| \leq \widetilde{\varphi}(t) \leq \widetilde{\varphi}(\beta),\quad
 \text{for $0\leq t<\beta$}.
 \end{equation}
  The second inequality is a consequence of the monotonicity of $\widetilde{\varphi}$ (recall that
 $g\geq 0$ on its domain, thanks to \eqref{ineg}).

 We now claim that $\varphi(t)$ satisfies the Cauchy condition for the existence of its limit as $t \rightarrow \beta^-$. This can be easily obtained by means of \eqref{ineg}, which (by using Volterra integral equations) ensures that\footnote{\color{blue} See Appendix \ref{Sec::Appendix1} for every detail.}
 $$\|\varphi(t)-\varphi(s)\|\leq |\widetilde{\varphi}(t)-\widetilde{\varphi}(s)|\qquad
 \forall\,\,t,s\in [0,\beta). $$
 Since $E$ is complete, it is possible to define
 $$\varphi(\beta) := \lim_{t \rightarrow \beta^-}\varphi(t).$$
 In particular (see also \eqref{diseg.qqqq}),  we deduce that
 $\|\varphi(\beta)\| \leq \widetilde{\varphi}(\beta)$ as the point $(\beta, \varphi(\beta))$ actually belongs to
   $\Omega$ (since $(\beta, \widetilde{\varphi}(\beta)) \in \tild{\Omega}$).

   Moreover, due to \eqref{EDOfifi}, the function $\varphi(t)$ possesses its right-derivative at $t=\beta$, since $f(t,y)$ is continuous and $(\beta, \varphi(\beta))\in\Omega$.
   Again from the fact that $(\beta,\varphi(\beta)) \in \Omega$, there exists a function $\varphi_1$ which solves the differential equation $y' = f(t,y)$ and such that $\varphi_1(\beta) = \varphi(\beta)$. The function $\varphi_1$ is defined on an open set containing $\beta$, and it permits to prolong $\varphi$. This clearly contradicts the maximality of $\varphi$ on its domain, and thus $\tild{\beta} \leq \beta$.

   Finally, the change of variable $t \mapsto -t$ (and the ``evenly'' nature of hypothesis
    \eqref{ineg}) allows us to prove that, with the above argument, $\tild{\beta} \leq -\alpha$.
\end{proof}
%%%%%%%%%%%%%%%%%%%%%%%%%%%%%%
\subsection{Application to the convergence domain of the Campbell-Hausdorff series.}\label{sec:applic}
 In the context of the Banach-Lie algebra $\Lieg$ of a Banach-Lie group (as in the previous sections),
 given $a,b$ sufficiently near $0$ in $\Lieg$,
 let us consider the function (see also \eqref{Hfunct})
 $$t \mapsto \phi(t) := \phi(a,t\,b), \quad \text{for $t$ in some neighborhood of $0\in \mathbb{R}$,}$$
 $\phi(t)$ taking its values in $\Lieg$. This function is differentiable, and
 (see also \eqref{S1IntCauchypb})
$$\begin{cases}
 \displaystyle \phi'(t) = \frac{\adj{\phi(t)}}{1 - \exp[-\adj{\phi(t)}]}
 \circ\frac{1 - \exp(-\adj{(tb)})}{\adj{(tb)}}(b) = \frac{\adj{\phi(t)}}{1 - \exp[-\adj{\phi(t)}]}(b) \\
 \phi(0) = a.
\end{cases}$$
 So $\phi(t)$ solves the Cauchy problem
 \begin{equation*}
 \begin{cases}
 y' = f(t,y) \\
 y(0) = a,
 \end{cases}\quad \text{with $f(t,y)=\dfrac{\adj{y}}{1 - \exp[-\adj{y}]}(b)$.}
\end{equation*}
 We denote by $\varphi(t)$ the maximal solution of the above Cauchy problem, for any given
  $(a,b)\in \Lieg\times \Lieg$ such that $\|a\|<2\pi$.
 In order to apply the results in the previous section, it is sufficient to estimate
 $$\left\|\frac{\adj{y}}{1 - \exp[-\adj{y}]}(b)\right\|,$$
 with the aim to obtain a majorizing function of the form $g(\|y\|)$ (see also \eqref{ineg}), such that
 $g(r)$ is non-decreasing for suitable positive values of $r$.
\begin{quote}
 \textcolor[rgb]{0.00,0.00,1.00}{[At this point, Mérigot uses some standard complex-variable
 technique to
 represent (in a quite explicit way, by exploiting Riemann's Zeta function\footnote{
 \color{blue}Indeed,
 it is a known fact from basic special-function theory that
 $$T_{2k}=\frac{B_{2k}}{(2k)!}=\frac{2(-1)^{k+1}}{(2\pi)^k}\,\zeta(2k),\quad k\in \mathbb{N}, $$
 where $\displaystyle\zeta(2k)=\sum_{n=1}^\infty \frac{1}{n^{2k}}$.
 Mérigot derives this very formula by standard residue-calculus technique.})
  the coefficients of the series expansion of
$$T(w):=\frac{w}{1 - \exp(-w)} = \sum_{n = 0}^{\infty}T_n\,w^n,\quad |w|<2\pi.$$
  The function $T$ is usually referred to, in the literature, as Todd's function.
    Compared to our previous notation in the proof of Proposition \ref{prop.bern}, one has
  $T(w)=-f(-w)$ and
   $T_n=-(-1)^n\alpha_n=(-1)^nB_n/n!$, where the $B_n$ are the Bernoulli numbers. We skip this part of the proof, which is not used elsewhere.]}
\end{quote}
 With the aim previously declared, let $G$ be the function defined as follows
 $$G(r) := 1 + \frac{r}{2} + \sum_{k = 1}^{\infty}|T_{2k}|\,r^{2k},\qquad r\in (-2\pi,2\pi),$$
 where the coefficients $T_{2k}$ come from the Maclaurin expansion of Todd's function\footnote{
 \color{blue}A simple calculation shows that
 $$\frac{r}{1 - \exp(-r)}=\frac{r}{2}+\frac{r/2}{\tanh(r/2)}, \qquad r\in \mathbb{R}. $$}
 $$T(r):=\frac{r}{1 - \exp(-r)} = 1+\frac{r}{2}+\sum_{k = 1}^{\infty} T_{2k}\,r^{2k}, \qquad r\in (-2\pi,2\pi). $$
\begin{quote}
 \color{blue}
 [With a standard trick of substituting $r$ with $ir$, \and takin into account known properties
 of the alternating sign of the coefficients $T_{2k}$, one finds that
 $$G(r)=2 + \frac{r}{2}\left(1 - \cot\left(\frac{r}{2}\right)\right), \qquad r\in (-2\pi,2\pi).] $$
\end{quote}
 One has (by using \eqref{submult})
 \begin{align*}
 &\left\|\frac{\adj{y}}{1 - \exp[-\adj{y}]}(b)\right\|
 = \left\| T(\adj{y})(b)\right\|=\left\|
 b+\frac{[y,b]}{2}+\sum_{k = 1}^{\infty} T_{2k}\,(\adj{y})^{2k}(b)
  \right\|\\
  &\leq \|b\|\left(1+\frac{\|y\|}{2}+\sum_{k = 1}^{\infty} |T_{2k}|\,\|y\|^{2k}\right)
  =\|b\|\,G(\|y\|)=:g(\|y\|).
\end{align*}
 The fact that $g$ is non-decreasing on $[0,2\pi)$ is an obvious consequence of the
 non-negativity of the coefficients in the expansion of $G$.

 In view of Proposition \ref{S1SS1MainProp}, let us consider the \emph{autonomous} Cauchy problem
$$\begin{cases}
 \widetilde{y}'= g(\widetilde{y})=\|b\|\,G(\widetilde{y}) \\
 \widetilde{y}(0) = \|a\|.
\end{cases}$$
 By the theory of separable ODEs, one knows that the supremum of the maximal domain of the
 maximal solution of the above Cauchy problem is
 $$\widetilde{\beta}=\int_{\|a\|}^{2\pi} \frac{1}{g(u)}\,\d u=
 \frac{1}{\|b\|}\int_{\|a\|}^{2\pi} \frac{1}{G (u)}\,\d u.$$
 We observe that $1<\widetilde{\beta}$
  if and only if
  $$\|a\|<2\pi\quad \text{and}\quad \|b\|<\int_{\|a\|}^{2\pi} \frac{1}{G (u)}\,\d u. $$
 Due to Proposition \ref{S1SS1MainProp}, time $t=1$ also belongs to the maximal domain
  of the maximal solution $\varphi(t)$. By using standard arguments\footnote{\color{blue}
  These arguments are not given in \cite{Merigot}; see our Appendix \ref{Sec::Appendix2}.}
  of analytic ODEs in Banach spaces, this also proves the convergence of the series representing $\phi(1)$, which
  is nothing but the Campbell-Hausdorff series \eqref{Hfunct} related to $(a,b)$.
%%%%%%%%%%%%%%%%%%%%%%%%%%%%%%%%%%%%%%%%%%%%%%%%%%%%%%%%%%%%%
\appendix
\section{Appendix - Study of the lifetime of the solution of a Cauchy problem.}\label{Sec::Appendix1}
 As anticipated in the Introduction, the aim of this first
 Appendix is to give a detailed proof of a slightly modified version
 of (Mérigot's) Proposition \ref{S1SS1MainProp}. As did by Mérigot in his manuscript,
 we split this proof into some preliminary lemmas, in order to deepen the analysis.
 Let us start by fixing the main notations and the hypotheses we shall assume throughout this section.

\begin{Def}\label{def.iniziale}
 We give the following definitions.
\begin{enumerate}
 \item
 Let $T$ be a fixed positive real number, or $T = \infty$, and let $(a,b)$ be an open interval of $\mathbb{R}$
 containing the origin (hence $a<0<b$). We denote by $\tild{\Omega}(T,a,b)$ the strip in $\mathbb{R}^2 \equiv \mathbb{R}_t\times\mathbb{R}_z$ defined as follows:
 $$\tild{\Omega}(T,a,b) := (-T,T)\times(a,b).$$

\item
 Let $(X,\|\cdot\|)$ be (real) Banach space. We denote by $\Omega(T,b)$ the subset of $\mathbb{R}_t\times X$ defined as follows:
 $$\Omega(T,b) := \left\{(t,y) \in \mathbb{R}\times X : (|t|,\|y\|) \in \tild{\Omega}(T,a,b)\right\}.$$
 Recalling the definition of $\tild{\Omega}(T,a,b)$, we can write (since $a<0$)
 $$\Omega(T,b) = \Big\{(t,y) \in \mathbb{R}\times X : |t| < T, \|y\| < b \Big\} = (-T,T)\times D(0,b),$$
 where, here and henceforth, $D(x,r)$ denotes the open ball in the normed space $X$ with center $x\in X$ and radius $r>0$.

\item
 $f: \Omega(T,b) \longrightarrow X$. We say that $f(t,y)$ satisfies the
 \textbf{hypothesis (LC)} if $f$ is locally Lipschitz continuous w.r.t.\,$y$ on $\Omega(T,b)$.
 This means that $f$ is continuous on its domain and that, for every compact subset $K$ contained in $\Omega(T,b)$, there exists a positive real constant $C = C(K)$ such that
 $$\|f(t,y) - f(t,y')\| \leq C\,\|y - y'\|, \quad \forall \,\, (t,y),\, (t,y') \in K.$$

\item
$g: \tild{\Omega}(T,a,b) \longrightarrow \mathbb{R}$.
 We say that $g$ satisfies the \textbf{hypothesis (MJ) w.r.t.\,$f$} (MJ stands for ``majorization''),  if the following three properties hold true:
\begin{itemize}
 \item [(i)] $g(t,z)$ is locally Lipschitz continuous w.r.t.\,$z$;

 \item [(ii)] for every fixed $t \in [0,T)$, the function
 $$[0,b) \ni z \mapsto g(t,z)$$
 is non-decreasing on its domain;

 \item [(iii)] one has the majorization
 \begin{subequations}
 \begin{equation}\label{MJ}
  \|f(t,y)\| \leq g(|t|,\|y\|), \quad \text{for every $(t,y) \in \Omega(T,b)$.}
 \end{equation}
 If this inequality is \emph{strict}, we say that $g$ satisfies the  \textbf{hypothesis (SMJ) w.r.t.\,$f$}:
 \begin{equation}\label{SMJ}
  \|f(t,y)\| < g(|t|,\|y\|), \quad \text{for every $(t,y) \in \Omega(T,b)$.}
 \end{equation}
 \end{subequations}
\end{itemize}
\end{enumerate}
\end{Def}
 \begin{Oss}\label{rem.nong}
 We explicitly remark that, if $g$ satisfies the  hypothesis (MJ) w.r.t.\,$f$, then
 $$g(t,z) \geq 0, \quad \forall \ (t,z) \in [0,T)\times[0,b).$$
 We are now in a position to start with the proof of the preliminary lemmas mentioned above, which will lead us to a quite easy proof of the main result of this section.
\end{Oss}
 All previous notations apply without the need to recall them every time.
\begin{Lm} \label{AppS1Lm1}
 Let us consider a continuous function
 $$g: \tild{\Omega}(T,a,b) \longrightarrow \mathbb{R},\qquad g=g(t,z),$$
 which is locally Lipschitz continuous w.r.t.\,$z$, and such that
\begin{equation}\label{posggg}
 g(t,z) > 0, \quad \forall \ (t,z) \in [0,T)\times[0,b).
\end{equation}
 Let us consider the Cauchy problem
$$\begin{cases}
z' = g(t,z) \\
z(0) = 0,
\end{cases}$$
and let
$\psi: (\tild{\alpha},\tild{\beta}) \longrightarrow \mathbb{R}$
 be its maximal solution. Then $\psi$ is strictly increasing on $[0,\tild{\beta})$ and
 in particular $\psi(t) > 0$, for every $t \in (0,\tild{\beta})$.
\end{Lm}
\begin{proof}
 First of all, since
 $\psi'(0) =g(0,\psi(0))= g(0,0) > 0$ (due to hypothesis \eqref{posggg}),
 there exists a positive number $\delta \leq \tild{\beta}$ such that
 $\psi'(t) > 0$, for every $t \in [0,\delta)$.
 Hence $\psi$ is strictly increasing on $[0,\delta)$
 and
\begin{equation*}
    \text{ $\psi(t) > \psi(0) = 0$, for every $t \in (0,\delta)$.}
\end{equation*}
 Let us consider the following number
 $$T:=\sup\Big\{s\in (0,\widetilde{\beta}]\,\,:\,\, \psi'(t)>0 \,\,\text{on $[0,s)$}\Big\}. $$
 Note that $T\geq \delta$. By definition of supremum,  it is not difficult to recognize that $\psi'>0$ on $[0,T)$.
 If $T=\widetilde{\beta}$, there is nothing left to prove.
 Let us suppose by contradiction that $T<\widetilde{\beta}$.
 Since, in this case, $T\in [\delta,\tild{\beta})$,
 the strict monotonicity of $\psi$ on $(0,T)$ ensures that
 $\psi(T)$ is finite and positive. Hence
 $\psi'(T) =g(T,\psi(T))> 0$.
 A continuity argument ensures that $\psi'>0$ on an interval $[T,T+\epsilon)$ (for some $\epsilon>0$), which is
 clearly in contradiction with the definition of $T$. This ends the proof.
\end{proof}
\begin{Lm} \label{AppS1Lm2}
 Let
 $f: \Omega(T,b) \longrightarrow X$
 satisfy the hypothesis (LC), and let
 $g: \tild{\Omega}(T,a,b) \longrightarrow \mathbb{R}$
 satisfy the hypothesis (SMJ) w.r.t.\,$f$. Let us consider the two following Cauchy problems
  (the first is valued in $X$, the second in $\mathbb{R}$)
$$\begin{cases}
y' = f(t,y) \\
y(0) = 0,
\end{cases} \qquad
\begin{cases}
z' = g(t,z) \\
z(0) = 0,
\end{cases}$$
 and let,  respectively,
$$\varphi: (\alpha,\beta) \longrightarrow X, \qquad \psi: (\tild{\alpha},\tild{\beta}) \longrightarrow \mathbb{R}$$
 be the two maximal solutions of these problems.
 Then one has
\begin{equation} \label{AppS1EqStatLm2}
\|\varphi(t)\| < \psi(t), \quad \forall \ t \in (0,\beta)\cap(0,\tild{\beta}).
\end{equation}
\end{Lm}
\begin{proof}
 Let us consider the following map
$$h: \Omega(T,b)\times\tild{\Omega}(T,a,b) \longrightarrow \mathbb{R}, \quad h(t,y,s,z) := g(s,z) - \|f(t,y)\|.$$
 Since $f$ satisfies the  hypothesis (LC) and $g$ satisfies the  hypothesis (SMJ) w.r.t.\,$f$, the func\-tion $h$ is continuous on its domain and
$$h(0,0,0,0) = g(0,0) - \|f(0,0)\| \stackrel{\eqref{SMJ}}{>} 0.$$
 By continuity, one can then find a positive real number, say $\tau$, such that
$$h(t,y,s,z) = g(s,z) - \|f(t,y)\| > 0,$$
 for all couples $(t,y) \in \Omega(T,b)$ and  $(s,z) \in \widetilde{\Omega}(T,a,b)$ such that
$$|t|,\, |s|,\, \|y\|,\, |z|\in [0,\tau).$$
 Summing up,
 \begin{equation}\label{riassu}
     \|f(t,y)\| < g(s,z),\quad \text{whenever $0\leq |t|,\, |s|,\,\|y\|,\, |z| < \tau$.}
 \end{equation}
 Let us now consider the two functions $\varphi$ and $\psi$.
 Due to \eqref{SMJ}, we are entitled to apply Lemma \ref{AppS1Lm1} and infer that $\psi$ is non-negative and strictly increasing on $[0,\tild{\beta})$.
 Another continuity argument ensures that, since
 $$\varphi(0) =0\in X,\quad \text{and}\quad \psi(0) = 0\in\mathbb{R},$$
 there exists a positive real number $\tau_1$, which we can assume to be smaller than $\tau$, such that
\begin{equation}\label{riassu2}
 0\leq \|\varphi(t)\|,\,\psi(t) < \tau, \qquad \forall\,\, t \in [0,\beta)\cap[0,\tild{\beta}),\,\,\,t < \tau_1.
\end{equation}
 Let us fix any arbitrary $t_0 \in [0,\beta)\cap[0,\tild{\beta})$ such that $0 < t_0 < \tau_1.$
 By the well-known Mean Value Theorem (for Banach spaces), there exist $\xi,\xi' \in (0,t_0)$ such that
\begin{align}
 \|\varphi(t_0)\| & = \|\varphi(t_0) - \varphi(0)\| \leq \|\varphi'(\xi)\|\cdot t_0 = \|f(\xi,\varphi(\xi))\|\cdot t_0, \label{MEANNN}\\
 \psi(t_0) & = \psi(t_0) - \psi(0) = \psi'(\xi')\cdot t_0 = g(\xi',\psi(\xi'))\cdot t_0,\label{MEANNN3}
\end{align}
 and thus (gathering together \eqref{riassu} and \eqref{riassu2} and the fact that $\tau_1<\tau$)
\begin{align*}
\|\varphi(t_0)\| \stackrel{\eqref{MEANNN}}{\leq} \|f(\xi,\varphi(\xi))\|\cdot t_0 \stackrel{\eqref{riassu}}{<} g(\xi', \psi(\xi'))\cdot t_0 \stackrel{\eqref{MEANNN3}}{=} \psi(t_0).
\end{align*}
 Note the crucial r\^ole of the equality in \eqref{MEANNN3}.

 By the arbitrariness of $t_0\in (0,\min\{\beta,\widetilde{\beta},\tau_1\})$, we have
\begin{equation*}
 \|\varphi(t)\| < \psi(t), \quad \forall \ t \in (0,\beta)\cap(0,\tild{\beta}),\,\,\,t < \tau_1.
\end{equation*}
 We are now ready to derive the desired inequality \eqref{AppS1EqStatLm2},
 essentially as in the proof of Lemma \ref{AppS1Lm1}.
 Indeed, let us assume that there exists a real number $t_0 \in (0,\beta)\cap(0,\tild{\beta})$ such that
\begin{equation} \label{cresce}
 \|\varphi(t)\| < \psi(t), \quad \forall \ t \in (0,t_0),
\end{equation}
 and let us prove that this inequality also holds true for $t = t_0$. For every $t \in (0,t_0)$ one has
\begin{gather}\label{MEANNN2}
 \begin{split}
 \|\varphi(t)\| &=\bigg\| \varphi(0)+\int_0^t f(u,\varphi(u))\,\d u\bigg\|
   \leq \int_0^t\|f(u,\varphi(u))\|\,\d u \\
   &\stackrel{\eqref{SMJ}}{<} \int_0^tg(u,\|\varphi(u)\|)\,\d u \leq  \int_0^tg(u,\psi(u))\,\d u = \psi(t).
\end{split}
\end{gather}
 In the second ``$\leq$'' inequality sign we used \eqref{cresce} and the non-decreasing monotonicity of $g$ w.r.t.\,its second variable
 (which is a part of hypothesis (SMJ)).

  Letting $t \rightarrow t_0^-$, the strict inequality  in \eqref{MEANNN2} still remains true (since this involves Riemann integrals
  of continuous functions), so that one has
$$\|\varphi(t_0)\| \leq \int_0^{t_0}\|f(u,\varphi(u))\|\,\d u < \int_0^{t_0}g(u,\|\varphi(u)\|)\,\d u \leq  \int_0^{t_0}g(u,\psi(u))\,\d u = \psi(t_0).$$
 Hence \eqref{cresce} is valid up to $t_0$ comprised.
 Due to the connectedness of $(0,\beta)\cap(0,\tild{\beta})$ (and the continuity of $\|\varphi(t)\|$, $\psi(t)$), the proof of the lemma
 easily follows.
\end{proof}
 We now want to remove the strict assumption from hypothesis (SMJ) in Lemma
  \ref{AppS1Lm2}.
\begin{Lm} \label{AppS1Lm3}
Let
 $f: \Omega(T,b) \longrightarrow X$
 satisfy the (LC) hypothesis and let $g: \tild{\Omega}(T,a,b) \longrightarrow \mathbb{R}$
 satisfy the (MJ) hypothesis  w.r.t.\,$f$. Let us consider the two following Cauchy problems
$$\begin{cases}
y' = f(t,y) \\
y(0) = 0,
\end{cases}
\qquad
\begin{cases}
z' = g(t,z) \\
z(0) = 0,
\end{cases}$$
and let,  respectively,
$$\varphi: (\alpha,\beta) \longrightarrow X, \qquad \psi: (\tild{\alpha},\tild{\beta}) \longrightarrow \mathbb{R}$$
 be the two maximal solutions of these problems.
Then one has
\begin{equation} \label{S1EqStatLm3}
\|\varphi(t)\| \leq \psi(t), \quad \forall \ t \in [0,\beta)\cap[0,\tild{\beta}),
\end{equation}
and thus, in particular,
$$\psi(t) \geq 0, \quad \forall \ t \in [0,\beta)\cap[0,\tild{\beta}).$$
\end{Lm}
\begin{proof}
 Let us fix a real number $\epsilon$ and let $h_{\epsilon}$ be the function defined as follows
 $$h_{\epsilon}: \tild{\Omega}(T,a,b) \longrightarrow \mathbb{R}, \quad h_{\epsilon}(t,z) := g(t,z) + \epsilon.$$
It is immediate to see that, if $\epsilon > 0$, the function $h_\epsilon$ satisfies the hypothesis (SMJ) w.r.t.\,$f$, and thus, denoting by $\psi_{\epsilon}$ the maximal solution of the Cauchy problem
$$\begin{cases}
z' = h_\epsilon(t,z) \\
z(0) = 0,
\end{cases}$$
defined on the open interval $(\tild{\alpha}_{\epsilon}, \tild{\beta}_{\epsilon})$, due to Lemma \ref{AppS1Lm2} we have
\begin{equation} \label{AppS1EqProofLm3}
\|\varphi(t)\| < \psi_{\epsilon}(t), \quad \forall \ t \in (0, \beta) \cap (0,\tild{\beta}_{\epsilon}),
\end{equation}
 for any arbitrary $\epsilon > 0$.
 Let now $t_0$ be arbitrarily fixed in $(0,\beta)\cap(0,\tild{\beta})$. Since the function $h_{\epsilon}$ is continuous w.r.t.\,the variable $\epsilon$, from well known
 general ODE theory, we know that the function
 $\epsilon \mapsto \tild{\beta}_{\epsilon}$
 is lower semi-continuous, and thus
$$\tild{\beta} = \tild{\beta}_0 \leq \liminf_{\epsilon \rightarrow 0}\tild{\beta}_{\epsilon}.$$
 From the fact that $t_0 < \tild{\beta}$, there follows the existence of a positive $\epsilon_1 = \epsilon_1(t_0)$ such that
 $$t_0 \in (0,\beta)\cap(0,\tild{\beta}_{\epsilon}), \quad \text{for every $\epsilon$ such that $|\epsilon| < \epsilon_1$,}$$
and thus, from \eqref{AppS1EqProofLm3}, we get
 $$\|\varphi(t_0)\| < \psi_{\epsilon}(t_0), \quad \text{whenever $0 < \epsilon < \epsilon_1$.}$$
 Letting $\epsilon \rightarrow 0^+$, since $\psi_{\epsilon}$ depends continuously on the parameter $\epsilon$, we get
$$\|\varphi(t_0)\| \leq \lim_{\epsilon \rightarrow 0^+}\psi_{\epsilon}(t_0) = \psi(t_0).$$
 From the arbitrariness of $t_0\in (0,\beta)\cap(0,\tild{\beta})$, and as $\|\varphi(0)\|= \psi(0) = 0$, we obtain \eqref{S1EqStatLm3}.
\end{proof}
 We are now in a position to state and prove the main result of this Appendix.
\begin{Th} \label{AppS1MainTh}
 Let
 $f: \Omega(T,b) \longrightarrow X$
 satisfy hypothesis (LC) and let
 $g: \tild{\Omega}(T,a,b) \longrightarrow \mathbb{R}$
 satisfy  hypothesis (MJ) w.r.t.\,$f$. Let us consider the two following Cauchy problems
$$\begin{cases}
y' = f(t,y) \\
y(0) = 0,
\end{cases} \qquad
\begin{cases}
z' = g(t,z) \\
z(0) = 0,
\end{cases}$$
 and let, respectively,
 $$\varphi: (\alpha,\beta) \longrightarrow X, \qquad \psi: (\tild{\alpha},\tild{\beta}) \longrightarrow \mathbb{R}$$
 be the two maximal solutions of these problems. Then one has
\begin{itemize}
\item [\emph{(i)}] the interval $[0,\tild{\beta})$ is contained in $[0,\beta)$;
\item [\emph{(ii)}] for every $t \in [0,\tild{\beta})$, one has
$$\|\varphi(t)\| \leq \psi(t).$$
\end{itemize}
\end{Th}
\begin{proof}
 We prove assertion (i) by contradiction. Let us assume that $\beta<\tild{\beta}$.
 This means that the intersection of $[0,\beta)$ and $[0,\tild{\beta})$ equals $[0,\beta)$, and thus, due to Lemma \ref{AppS1Lm3}, we have
\begin{equation} \label{AppS1Eq1ProofMainTh}
\|\varphi(t)\| \leq \psi(t) \quad \forall \ t \in [0, \beta).
\end{equation}
 Moreover, due to hypothesis (MJ), from Remark \ref{rem.nong} we get
 \begin{equation}\label{gggdereq}
    g(t,z) \geq 0, \quad \forall \ (t,z) \in [0,T)\times[0,b).
 \end{equation}
 Since $\psi$ is non-negative on $[0,\beta)$ (see \eqref{AppS1Eq1ProofMainTh}), inequality \eqref{gggdereq} jointly with
 the ODE solved by $\psi$ (i.e., $\psi'(t)=g(t,\psi(t))$), we get $\psi'(t)\geq 0$ on
 $[0,\beta)$, and thus (note that $\beta$ belongs to the maximal domain of $\psi$ by our assumption $\beta<\widetilde{\beta}$)
\begin{equation} \label{AppS1Eq2ProofMainTh}
\psi(t) \leq \psi(\beta), \quad \forall \ t \in [0,\beta).
\end{equation}
Let now $\epsilon$ be a positive real number such that the \emph{compact} neighborhood of $\beta$
 $$I_{\epsilon}(\beta) := [\beta-\epsilon, \beta+\epsilon]$$
 is contained in $[0,\tild{\beta})$. For every $s,t \in I_{\epsilon}(\beta)$ such that $t,s < \beta$, one has
\begin{align*}
 \|\varphi(t) - \varphi(s)\| & \leq \left|\int_t^s\|f(u,\varphi(u))\|\,\d u\right|
 \stackrel{\eqref{MJ}}{\leq}
 \left|\int_t^s g(u,\|\varphi(u)\|)\,\d u\right| \\
 &\leq \left|\int_t^sg(u,\psi(u))\,\d u\right|
 =\left|\int_t^s \psi'(u)\,\d u\right|
  =|\psi(t)-\psi(s)| \\
  &\leq \left(\max_{[\beta-\epsilon,\beta]}|\psi'|\right)\cdot |t - s|.
\end{align*}
 In the third ``$\leq$'' sign, we used \eqref{AppS1Eq1ProofMainTh}, together with the non-increasing
 monotonicity of $g$ w.r.t.\,its second argument, which is part of the (MJ) hypothesis.

 With the clear meaning of the following symbol, we have the Cauchy condition
$$\lim_{s,t \rightarrow \beta^-}\|\varphi(t) - \varphi(s)\| = 0.$$
 Since the space $X$ is complete, this grants the existence of  $\lim_{t \rightarrow \beta^-}\varphi(t)$ in $X$. We set
$$\varphi(\beta) := \lim_{t \rightarrow \beta^-}\varphi(t).$$
 From inequality \eqref{AppS1Eq1ProofMainTh}, it follows that
 $\|\varphi(\beta)\| \leq \psi(\beta)$, and thus, since the point $(\beta,\psi(\beta))$
 belongs to $\tild{\Omega}(T,a,b)$, the point $(\beta, \varphi(\beta))$ belongs to $\Omega(T,b)$.
 Moreover, from the fact that $f$ is continuous on $\Omega(T,b)$, it follows that the function $\varphi$ has the right-derivative at $\beta$, and
$$\varphi'(\beta) = \lim_{t \rightarrow \beta^-}\varphi'(t) = \lim_{t \rightarrow \beta^-}f(t,\varphi(t)) = f(\beta,\varphi(\beta)).$$
Let us now consider the Cauchy problem
$$\begin{cases}
y' = f(t,y) \\
y(\beta) = \varphi(\beta),
\end{cases}$$
 which is well-posed since $(\beta,\varphi(\beta)) \in \Omega(T,b)$. There exists a local solution
 $\varphi_1$ of this problem, defined in a small neighborhood
 $(\beta - \delta,\beta + \delta)$ of $\beta$, and since $\varphi$ is a right-solution of this problem as well,  we
  obtain a prolongation of $\varphi$ beyond $\beta$, which is clearly a contradiction with the
  maximality of $\varphi$.

  This proves that $\widetilde{\beta} \leq  \beta$, and from inequality \eqref{S1EqStatLm3} in Lemma \ref{AppS1Lm3}
  we directly derive the proof of statement (ii) of the present theorem.
\end{proof}
We conclude this Appendix by proving a simple corollary of Theorem \ref{AppS1MainTh}, which essentially shows that the initial conditions
$$y(0) = 0, \qquad z(0) = 0,$$
can be replaced with any conditions of the form
$$y(0) = x, \qquad z(0) = \|x\|,$$
for some $x \in \Omega(T,b).$
\begin{Cor} \label{AppS1CorInCond}
Let
 $f: \Omega(T,b) \longrightarrow X$
 satisfy hypothesis (LC) and let
 $g: \tild{\Omega}(T,a,b) \longrightarrow \mathbb{R}$
 satisfy hypothesis (MJ) w.r.t.\,$f$. Let us consider, for a fixed $x \in \Omega(T,b)$, the Cauchy problems
$$\begin{cases}
y' = f(t,y) \\
y(0) = x,
\end{cases}
\qquad
\begin{cases}
z' = g(t,z) \\
z(0) = \|x\|,
\end{cases}$$
and let, respectively,
$$\varphi: (\alpha,\beta) \longrightarrow X, \qquad \psi: (\tild{\alpha},\tild{\beta}) \longrightarrow \mathbb{R}$$
 be their maximal solutions. Then one has
 $\tild{\beta}\leq \beta$ and $\|\varphi(t)\| \leq \psi(t)$, for every $t \in [0,\tild{\beta})$.
\end{Cor}
\begin{proof}
 Let us consider the maps $\tau_x$ and $\tau_{\|x\|}$ defined by
\begin{align*}
\tau_x : X \longrightarrow X, \quad & \tau_x(y) := y - x, \\
\tau_{\|x\|} : \mathbb{R} \longrightarrow \mathbb{R}, \quad & \tau_{\|x\|}(z) := z - \|x\|.
\end{align*}
 Let $A,\tild{A}$ be the subsets of $X$ and $\mathbb{R}$, respectively, given by
\begin{align*}
A & := \Big\{(t,y) \in \mathbb{R}\times X : (t,\tau_x^{-1}(y)) \in \Omega(T,b)\Big\}, \\
\tild{A} & := \Big\{(t,z) \in \mathbb{R}\times \mathbb{R} : (t,\tau_{\|x\|}^{-1}(z)) \in \tild{\Omega}(T,a,b)\Big\}.
\end{align*}
We explicitly remark that, since $x \in \Omega(T,b)$, the origin of $\mathbb{R}^2$ belongs to $\tild{A}$ and the origin of $\mathbb{R}\times X$ belongs to $A$. Let now $f_1, g_1$ be the two functions defined as follows
\begin{align*}
f_1: A \longrightarrow X, \quad & f_1(t,y) := f(t,y+x), \\
g_1 : A \longrightarrow \mathbb{R}, \quad & g_1(t,z) := g(t,z + \|x\|),
\end{align*}
and let us consider the two Cauchy problems
$$\begin{cases}
y' = f_1(t,y) \\
y(0) = 0,
\end{cases}
\qquad
\begin{cases}
z' = g_1(t,z) \\
z(0) = 0.
\end{cases}$$
It is immediate to recognize that the functions
\begin{align*}
\varphi_1 : (\alpha,\beta) \longrightarrow X, \quad & \varphi_1(t) := \varphi(t) - x, \\
\psi_1 :(\tild{\alpha},\tild{\beta}) \longrightarrow \mathbb{R}, \quad & \psi_1(t) := \psi(t) - \|x\|,
\end{align*}
 are the two maximal solutions of the above problems, respectively.
 We can apply Theorem \ref{AppS1MainTh} to $\varphi_1$ and to $\psi_1.$
 In order to do this, we have to check that all the hypotheses of this theorem are satisfied. First of all, we remark that
$$\tild{A} = \tild{\Omega}(T,a-\|x\|,b-\|x\|) \equiv \tild{\Omega}(T,a_1,b_1).$$
 On the other hand, the set $A$ is \emph{not} equal to $\Omega(T,a_1,b_1)$, but the following inclusion holds true:
 $A \supseteq \Omega(T,a_1,b_1)$.
 We can then consider the function $f_1$ \emph{restricted} to $\Omega(T,a_1,b_1)$, which we denote by $h$. One has:

\medskip$\bullet$\,\,the function $h$ satisfies hypothesis (LC);

\medskip$\bullet$\,\,the function $g_1$ satisfies hypothesis (MJ) w.r.t.\,$h$.
 Indeed, it is obvious that $g_1$ is locally Lipschitz continuous w.r.t.\,its second variable.
 For every fixed $t \in [0,T)$, the map
 $$[0,b_1) \ni z \mapsto g_1(t,z) = g(t,z+\|x\|)$$
 is non-decreasing (being $\|x\| \geq 0$), and for every $(t,y) \in \Omega(T,a_1,b_1)$ one has (since $\|y+x\|$ and $\|y\|+\|x\|$ are positive real numbers in $[0,b)$)
$$\|h(t,y)\| = \|f(t,y+x)\| \leq g(|t|,\|y+x\|) \leq g(|t|,\|y\|+\|x\|) = g_1(|t|,\|y\|).$$
 Finally, since the domain of the maximal solution $\phi$ of the Cauchy problem
$$\begin{cases}
y' = h(t,y) \\
y(0) = 0,
\end{cases}$$
is an open interval $I = (\xi,\eta)$ contained in $(\alpha,\beta)$, and since
 $\phi(t) = \varphi_1(t)$ for every $t \in I$,
it follows from Theorem \ref{AppS1MainTh} that
 $\widetilde{\beta}\leq \eta$ and $\|\phi(t)\| \leq \psi_1(t)$, for every $t \in [0,\widetilde{\beta})$.

 This yields $[0,\widetilde{\beta}) \subseteq [0,\beta)$,
 and
 \begin{align*}
\|\varphi(t)\| & \leq \|\varphi_1(t)\| + \|x\| = \|\phi(t)\| + \|x\| \leq \psi_1(t) + \|x\| = \psi(t),
\end{align*}
 for any $t \in [0,\widetilde{\beta})$. This ends the proof.
\end{proof}
%%%%%%%%%%%%%%%%%%%%%%%%%%%%%%%%%%%%%%%%5
\section{Appendix - Application to the convergence domain of the CBHD series.}
\label{Sec::Appendix2}
 We end the dissertation with a second Appendix, in which we want to show
 more closely how Theorem \ref{AppS1MainTh} can be used to obtain an enlarged domain for the (homogenous) Campbell-Baker-Hausdorff-Dynkin series.
 This Appendix B furnishes further details for the comprehension of Mérigot's manuscript \cite{Merigot}, whose
 contents are described in Section \ref{Sec::1}.

 First of all, let us recall the main definitions and notations we need for our purpose.

\medskip$\bullet$\,\,We denote by $\mathscr{T}_{\mathbb{R}}(x,y)$ the unital associative algebra of the polynomials (with coefficients in $\mathbb{R}$) in the two \emph{non-commuting} indeterminates $x$ and $y$. For every $(i,j) \in \mathbb{N}\times\mathbb{N}$ with $i+j \geq 1$, we define
$$Z_{i,j} := \sum_{n = 1}^{i+j}\frac{(-1)^{n+1}}{n}\cdot\sum_{
\begin{subarray}{c}
(i_1,j_1),\ldots,(i_n,j_n)\,\neq\,(0,0) \\
i_1+\cdots+i_n\,=\,i \\
j_1+\cdots+j_n\,=\,j
\end{subarray}}\frac{x^{i_1}y^{j_1}\cdots x^{i_n}y^{j_n}}{i_1!j_1!\ldots i_n!j_n!}.$$
 We also set $Z_{0,0} := 0$. The notation $Z_{i,j}(x,y)$ will also apply occasionally. Due to the classical algebraic version of
 the CBHD Theorem (see \cite[Chapter 3]{BonfiglioliFulci}), we know that
 the family of polynomials $\{Z_{i,j}\}_{i,j \in \mathbb{N}}$ is in fact a family of \emph{Lie-polynomials}, that is,
$$Z_{i,j} \in \Lie, \quad i,j \in \mathbb{N} .$$
 Here $\Lie$ denotes the smallest Lie subalgebra of $\mathscr{T}_{\mathbb{R}}(x,y)$ containing $x$ and $y$ ($\mathscr{T}_{\mathbb{R}}(x,y)$
 is equipped with the Lie-algebra structure naturally associated with its associative multiplication).

 \medskip$\bullet$\,\,Let $(X,[\cdot,\cdot],\|\cdot\|)$ be a real Banach-Lie algebra\footnote{We say that $X$ is a real Banach-Lie algebra if $X$
 is equipped with a (possibly infinite-dimensional) real Lie-algebra structure $(X,[\cdot,\cdot])$ and,
 at the same time, with a Banach-space structure $(X,\|\cdot\|)$ over $\mathbb{R}$,
 these structures being compatible, in that the map
 $X\times X\ni (g,g')\mapsto [g,g']\in X$ is required to be continuous.
 Since the bracket is a bilinear map, the above continuity assumption is equivalent to the existence
 of a positive constant $M$ such that $\|[g,g']\|\leq M\,\|g\|\,\|g'\|$ for
 every $g,g'\in X$. By replacing $\|\cdot\|$ with the equivalent
 norm $M\,\|\cdot\|$, we can suppose (and we shall do it henceforth)
 that the norm $\|\cdot\|$ is \emph{Lie-sub-multiplicative}, i.e.,
 \begin{equation}\label{Liesubmu}
    \|[g,g']\|\leq \|g\|\,\|g'\|,\quad\text{for every $g,g'\in X$.}
 \end{equation}} over the field $\mathbb{R}$.
For every $a,b \in X$, we denote by $u_{a,b}$ the \emph{unique} Lie-algebra morphism from $\Lie$ to $X$ such that
$$u_{a,b}(x) = a, \qquad u_{a,b}(y) = b.$$
For every $i,j \in \mathbb{N}$, the $(i,j)$-Dynkin polynomial (in $a$ and $b$) is the element of $X$ defined in the following way
$$Z_{i,j}(a,b) := u_{a,b}(Z_{i,j}(x,y)).$$
For every fixed $a,b \in X$, one has
$$\begin{cases}
Z_{1,0}(a,b) = a, \\
Z_{0,1}(a,b) = b, \\
Z_{i,0}(a,b) = Z_{0,i}(a,b) = 0, \quad \forall \ i \geq 2.
\end{cases}$$
Moreover, for every $i,j \in \mathbb{N}$ with $i+j \geq 1$, the following recursion formulas hold true:
\begin{gather}\label{ricorsive}
\begin{split}
Z_{i+1,j}(a,b) & = \frac{1}{i+1}\sum_{
\begin{subarray}{c}
1\,\leq\,h\,\leq\,i + j \\
(i_1,j_1),\ldots,(i_h,j_h)\,\neq\,(0,0) \\
i_1 + \cdots + i_h\,=\,i \\
j_1 + \cdots + j_h\,=\,j
\end{subarray}}K_h\cdot[Z_{i_1,j_1}(a,b), \ldots, [Z_{i_h,j_h}(a,b),a] \ldots ], \\
Z_{i,j+1}(a,b) & = \frac{1}{j+1}\sum_{
\begin{subarray}{c}
1\,\leq\,h\,\leq\,i + j \\
(i_1,j_1),\ldots,(i_h,j_h)\,\neq\,(0,0) \\
i_1 + \cdots + i_h\,=\,i \\
j_1 + \cdots + j_h\,=\,j
\end{subarray}}(-1)^hK_h\cdot[Z_{i_1,j_1}(a,b), \ldots, [Z_{i_h,j_h}(a,b),b] \ldots ],
\end{split}
\end{gather}
where $\left\{K_n\right\}_{n \in \mathbb{N}}$ is the sequence in $\mathbb{Q}$ defined as follows
$$
K_0 := 1 ,\qquad
\displaystyle K_n := -\sum_{i = 0}^{n-1}\frac{K_i}{(n + 1 - i)!},\quad  n \geq 1.
$$
 For an algebraic proof of these facts, see \cite{BiagiBonfiglioli}.
 By using the well-known Bernoulli numbers $\{B_n\}_n$, one has $K_n=B_n/n!$, so that
 $$\sum_{n=0}^\infty K_n\, z^n=\frac{z}{e^z-1},\quad \forall\,\,z\in \mathbb{C}:\,\,|z|<2\pi. $$
 A simple calculation (based on the sign of the Bernoulli numbers) shows that
 $$\sum_{n=0}^\infty |K_n|\, z^n=
 2 + \frac{z}{2}\left(1 - \cot\left(\frac{z}{2}\right)\right)=:G(z),\quad \forall\,\,z\in \mathbb{C}:\,\,|z|<2\pi. $$

\medskip$\bullet$\,\,Let $\Delta$ be the neighborhood of $(0,0)$ in $X\times X$ defined as follows
$$\Delta := \big\{(a,b) \in X\times X: \|a\| + \|b\| < \log2\big\}$$
 It is well-known that (see, e.g., the pioneering work by Dynkin \cite{Dynkin}), for every $a,b \in X$ such that $(a,b) \in \Delta$, the (homogeneous) CBHD series
$$Z(a,b) := \sum_{n = 0}^{\infty}\bigg(\sum_{i+j = n}Z_{i,j}(a,b)\bigg)$$
 is convergent.
% Moreover, it is (totally and) uniformly convergent on every set of the type
%$$D_{\delta} := \big\{(a,b) \in X\times X : \|a\| + \|b\| \leq \delta\big\}, \quad \delta < \log2.$$
 Moreover, since $(0,0) \in \Delta$ and $Z(0,0) = 0$, it is possible to find a neighborhood $W$ of $(0,0)$ contained in $\Delta$ such that
\begin{equation}\label{2pibenposta}
 \|Z(a,b)\| < 2\pi, \quad \forall \ a, b \in X : (a,b) \in W.
\end{equation}
 With the above choice of $W$, let $(a,b) \in W$ and let $\epsilon$ be a positive real number such that $(a,t\,b) \in W$ for every $t \in (-\epsilon,\epsilon)$. The function
$$\gamma: (-\epsilon,\epsilon) \longrightarrow X, \quad \gamma(t) := Z(a,tb),$$
 is a well-defined ($X$-valued) real analytic function on $(-\epsilon,\epsilon)$, and it is a solution of the following Cauchy problem
$$\begin{cases}
 y' = G(-\adj y)(b)=\sum_{n = 0}^{\infty}K_n(-\adj{y})^n(b) \\
y(0) = a.
\end{cases}$$
 We remark that, due to \eqref{2pibenposta}, $G(-\ad \gamma(t))$ is well-posed for every $t\in (-\epsilon,\epsilon)$.

 We are now ready to determine a subset of $X\times X$, that \emph{strictly} contains the set $\Delta$,
 on which the CBHD series converges. Using the notations introduced in the previous sections, let us consider the set
 $$\tild{\Omega} := \tild{\Omega}(\infty,-2\pi,2\pi) = \mathbb{R}\times(-2\pi,2\pi),$$
 and let $\Omega$ be the subset of $\mathbb{R}\times X$ associated with $\tild{\Omega}$, that is,
 $$\Omega := \Omega(\infty,2\pi) = \mathbb{R}\times \{y\in X\,:\,\|y\|<2\pi\}.$$
 For a fixed $b \in X$, we define the function $f_b$ in the following way
 $$f_b : \Omega \longrightarrow X, \quad f_b(t,y) :=G(-\adj y)(b)= \sum_{n = 0}^{\infty}K_n(-\adj{y})^n(b).$$
 [As a fact, note that $f_b$ does not depend on $t$.]
 We explicitly remark that this definition is well-posed, since the complex power series
 $\sum_{n = 0}^{\infty}K_n\,z^n$
 has radius of convergence $2\pi$. Since, for every $n \geq 1$, the map
 $$X^n \ni (y_1,\ldots,y_n) \mapsto [y_1, \ldots, [y_n,b] \ldots]$$
 in $n$-linear, the function $f_b$ is infinitely Fréchet-differentiable on $\Omega$, and in particular it is locally Lipschitz continuous on the same set. For every fixed $a\in X$ such that $\|a\|<2\pi$, we can then consider the following Cauchy problem
\begin{equation} \label{AppS2EqZCauchyPb}
\begin{cases}
 y' = f_b(t,y) \\
 y(0) = a,
\end{cases}
\end{equation}
 which has a maximal solution $\varphi_{a,b}$ defined on a open interval $\mathcal{D}_{a,b}$ containing $0$. We now want to apply Theorem \ref{AppS1MainTh} from Appendix A (or, more precisely, Corollary \ref{AppS1CorInCond}) to this Cauchy problem. In order to do this, we have to find a function $g_b$, defined on $\tild{\Omega}$, which satisfies the hypothesis (MJ) w.r.t.\,$f_b$ (see Definition \vref{def.iniziale}). Since $f_b$ is represented by a convergent series of functions and for every $n \geq 1$ one has (see \eqref{Liesubmu} in the footnote on page \pageref{Liesubmu})
 $$\|(-\adj{y})^n(b)\| \leq \|b\|\cdot \|y\|^n, \quad \forall \ y \in X,$$
 the natural candidate for $g_b$ is the function
\begin{equation} \label{AppS2EqDefg_b}
 g_b: \tild{\Omega} \longrightarrow \mathbb{R}, \quad g_b(t,z) := \|b\|\sum_{n = 0}^{\infty}|K_n|\,z^n=
 \|b\| \Big(2 + \frac{z}{2}\left(1 - \cot\left(\frac{z}{2}\right)\right)\Big).
\end{equation}
 Let us check that this function satisfies the hypothesis (MJ) w.r.t.\,$f_b$.
\begin{Lm}\label{lemma.biagi2}
 The function $g_b$ defined in \eqref{AppS2EqDefg_b} satisfies the hypothesis (MJ) w.r.t.\,$f_b$.
\end{Lm}
\begin{proof}
 First of all, since $g_b \in \mathrm{C}^{\infty}(\tild{\Omega};\mathbb{R})$ ($g_b$ is in fact real analytic on $\tild{\Omega}$), it is locally Lipschitz continuous on $\tild{\Omega}$. Moreover, for every fixed $t \in [0,+\infty)$, the function
 $[0,2\pi) \ni z \mapsto g_b(t,z) = \|b\|\sum_{n = 0}^{\infty}|K_n|z^n$
 is increasing. Finally, for every $(t,y) \in \Omega$, one has
     $$\|f_b(t,y)\| \leq \sum_{n = 0}^{\infty}|K_n|\cdot \|(-\adj{y})^n(b)\| \leq \|b\| \sum_{n = 0}^{\infty}|K_n|\cdot \|y\|^n= g_b(|t|,\|y\|).$$
 This ends the proof.
\end{proof}
 We can therefore apply Corollary \ref{AppS1CorInCond}: It ensures that, if we denote by $\psi_{a,b}$ the maximal solution of the Cauchy problem
\begin{equation} \label{AppS2EqRealCauchyPb}
\begin{cases}
 z' = g_b(t,z) \\
 z(0) = \|a\|,
\end{cases}
\end{equation}
 defined on the open interval $(\tild{\alpha}_{a,b},\tild{\beta}_{a,b})$, one has
 \begin{itemize}
   \item $[0,\tild{\beta}_{a,b}) \subseteq \mathcal{D}_{a,b}$;
   \item $\|\varphi_{a,b}(t)\| \leq \psi(t)$, for every $t \in [0,\tild{\beta}_{a,b})$.
 \end{itemize}
 What is crucial is that, since $g$ in independent of $t$, the ODE in \eqref{AppS2EqRealCauchyPb} is a separable equation, and thus we can determine explicitly the values of $\tild{\alpha}_{a,b}$ and $\tild{\beta}_{a,b}$:
 $$\tild{\alpha}_{a,b} = \frac{1}{\|b\|}\int_{\|a\|}^{-2\pi}\frac{1}{G(u)}\,\d u,
  \qquad \tild{\beta}_{a,b} = \frac{1}{\|b\|}\int_{\|a\|}^{2\pi}\frac{1}{G(u)}\,\d u.$$
The integrals written above are both finite, since
\begin{equation}\label{Gammmma.G}
 G(u) = 2 + \frac{u}{2}\left(1 - \cot\left(\frac{u}{2}\right)\right), \quad \forall \ u \in (-2\pi,2\pi),
\end{equation}
 and thus the function $1/g_b$ is a bounded continuous function on the open interval $(-2\pi,2\pi)$.
 Finally, let $\Gamma$ be the subset of $X\times X$ defined as follows
\begin{equation}\label{Gammmma}
 \Gamma := \left\{(a,b) \in X\times X :\quad  \|a\|<2\pi,\,\,\,\|b\| < \int_{\|a\|}^{2\pi}\frac{1}{G(u)}\,\d u\right\},    
\end{equation}
 and let us prove that the CBHD series is convergent on $\Gamma$. A similar argument proves the convergence on the set analogous to $\Gamma$, with $a$ and $b$
 interchanged. This gives a convergence domain analogous to that in \cite{BlanesCasas}, where the convergence result is proved in the finite-dimensional case.
\begin{Lm} \label{AppS2LmTaylorphi}
 With the above notation, for every $a,b \in X$ with $\|a\|<2\pi$, the power series
 $$\sum_{n = 0}^{\infty}\frac{1}{n!}\,\varphi_{a,b}^{(n)}(0)\,t^n$$
 is absolutely convergent on $[0,\tild{\beta}_{a,b})$, and its sum is equal to $\varphi_{a,b}(t)$ (the maximal solution of
 \eqref{AppS2EqZCauchyPb}) on $[0,\tild{\beta}_{a,b})$. Furthermore, for every $n\geq 1$,
\begin{gather}\label{AppS2EqDerphi_1}
\begin{split}
 \varphi_{a,b}'(0) & = b + \sum_{h = 1}^{\infty}(-1)^h K_h\,[a,\ldots[a,b]\ldots],  \\
 \frac{\varphi_{a,b}^{(n+1)}(0)}{(n+1)!} & = \tfrac{1}{n+1}\sum_{h = 1}^{\infty}\sum_{n_1+\cdots+n_h = n}\frac{(-1)^h K_h}{n_1!\cdots n_h!}
 \,[\varphi_{a,b}^{(n_1)}(0),\ldots[\varphi_{a,b}^{(n_h)}(0),b]\ldots].
\end{split}
\end{gather}
\end{Lm}
\begin{proof}
 Let $a,b \in X$ be fixed, with $\|a\|<2\pi$. For every $t \in [0,\tild{\beta}_{a,b}) \subseteq \mathcal{D}_{a,b}$, one has
 $$\varphi_{a,b}^{(k)}(0) = \frac{\d^{k-1}}{\d t^{k-1}}\bigg|_{t = 0}(f_b \circ \varphi_{a,b})(t), \quad \forall \ k \geq 1,$$
 and it is not difficult to prove that\footnote{On could use, for example, the Faà di Bruno's Formula, together with the majorizing property of $g_b$ w.r.t.\,$f_b$.
  Alternatively, one can solve \eqref{AppS2EqZCauchyPb} and \eqref{AppS2EqRealCauchyPb} by series, thus obtaining explicit inductive formulas for the derivatives
  of $\varphi_{a,b}^{(k)}$ and $\psi_{a,b}^{(k)}$, and then majorize directly.}
 $$\|\varphi_{a,b}^{(k)}(0)\| \leq \psi_{a,b}^{(k)}(0), \quad \forall \ k \in \mathbb{N}.$$
 This proves that the MacLaurin series of $\varphi_{a,b}$ is absolutely convergent on $[0,\widetilde{\beta}_{a,b})$, since
$$\sum_{n = 0}^{\infty}\frac{1}{n!}\,\|\varphi_{a,b}^{(k)}(0)\|\,t^n \leq \sum_{n = 0}^{\infty}\frac{\psi_{a,b}^{(n)}(0)}{n!}\,t^n = \psi_{a,b}(t), \quad \forall \ t \in [0,\tild{\beta}_{a,b}).$$
 Let us consider the sum of the MacLaurin series of $\varphi_{a,b}$, that is the function $u$ defined by
$$u: [0,\tild{\beta}_{a,b}) \longrightarrow X, \quad u(t) := \sum_{n = 0}^{\infty}\frac{1}{n!}\,\varphi_{a,b}^{(n)}(0)\,t^n.$$
 It is obvious that $u$ is a real analytic $X$-valued function, and
 $\|u(t)\| \leq \psi_{a,b}(t) < 2\pi$, for every $t \in [0,\tild{\beta}_{a,b})$.
 In order to show that $u$ is equal to $\varphi_{a,b}$, we prove that $u$ solves the Cauchy problem \eqref{AppS2EqZCauchyPb} as well. First of all, one has
 $u(0) = \varphi_{a,b}(0)$.
 Moreover, since the norm of $u(t)$ is less than or equal to $2\pi$, the composition $f_b \circ u$ is a well-defined analytic function, and one has (remembering that $u$ and $\varphi_{a,b}$ have the same derivatives at $0$, for any order)
 $$\frac{\d^n}{\d t^n}\bigg|_{t = 0}(f_b \circ u)(t) = \varphi_{a,b}^{(n + 1)}(0) = u^{(n+1)}(0), \quad \forall \ n \geq 0.$$
 This shows that the MacLaurin series of $u'$ and of $f_b \circ u$ coincide, and thus (by unique continuation)
 $u'(t) = f_b(u(t))$, for any $t \in [0,\tild{\beta}_{a,b})$.
 Finally, if we write down explicitly the expression of $(f_b \circ u)^{(n)}(0)$, then we obtain 
 \eqref{AppS2EqDerphi_1}. This ends the proof.
\end{proof}
\noindent Due to Lemmas  \ref{lemma.biagi2} and \ref{AppS2LmTaylorphi}, 
 we are in a position to prove the main result of Appendix B:
\begin{Th} \label{AppS2MainTh}
 Let $X$ be a Banach-Lie algebra over $\mathbb{R}$ (equipped with a Lie-sub-multiplicative norm $\|\cdot\|$), and let $\Gamma$ be as in 
 \eqref{Gammmma}, with $G$ as in \eqref{Gammmma.G}.

 Then the homogeneous CBHD series $\sum_{n = 0}^{\infty}\left(\sum_{i+j = n}Z_{i,j}(a,b)\right)$
 is convergent for every $(a,b)\in \Gamma$. An analogous result holds by interchanging the roles of $a$ and $b$.
\end{Th}
\begin{proof}
 We prove the theorem by showing that, for every $(a,b) \in \Gamma$, the \emph{double series}
 $\sum_{i,j = 0}^{\infty}\|Z_{i,j}(a,b)\|$
 is convergent. To this aim, we fist remark that, arguing by induction (by also taking into account the recursive 
 relations \eqref{ricorsive}), one can easily obtain the following estimates (for any detail, see \cite{BiagiBonfiglioli})
\begin{align*}
 \sum_{i = 0}^{\infty}\|Z_{i,1}(a,b)\| & \leq \|b\|\left(1 + \sum_{h = 1}^{\infty}|K_h|\|a\|^h\right) = g_b(\|a\|) = \psi_{a,b}'(0), \\
 \sum_{i = 0}^{\infty}\|Z_{i,j}(a,b)\| & \leq \frac{\|b\|}{j+1}\left(\sum_{h = 0}^{\infty}\sum_{n_1+\cdots+n_h = j}\frac{|K_h|}{n_1!\cdots n_h!}
 \|\varphi_{a,b}^{(n_1)}(0)\|\cdots\|\varphi_{a,b}^{(n_h)}(0)\|\right) 
 \leq \frac{\psi_{a,b}^{(j)}(0)}{j!}.
\end{align*}
 Since the set $\Gamma$ is precisely the subset of $X\times X$ consisting of all the couples $(a,b)$ such that $\|a\|<2\pi$ and 
 $\tild{\beta}_{a,b} > 1$ (recall that $\tild{\beta}_{a,b}$ is the supremum of the maximal domain of the solution $\psi_{a,b}$ of 
 \eqref{AppS2EqRealCauchyPb}), we have
$$\sum_{i,j = 0}^{\infty}\|Z_{i,j}(a,b)\| = \sum_{j = 0}^{\infty}\left(\sum_{i = 0}^{\infty}\|Z_{i,j}(a,b)\|\right) \leq \sum_{j = 0}^{\infty}\frac{\psi_{a,b}^{(j)}(0)}{j!} = \psi_{a,b}(1).$$
This ends the proof.
\end{proof}

\end{document}